\documentclass[12pt,reqno]{article}
\usepackage[utf8]{inputenc}
\usepackage{amsmath}
\usepackage{amsfonts}
\usepackage{amssymb}
\usepackage{amsthm}
\usepackage{cases}
\usepackage{cite}
\usepackage[colorlinks]{hyperref}
\usepackage{fancyref}

\allowdisplaybreaks 

\renewcommand{\oddsidemargin}{5mm}

\numberwithin{equation}{section}
\newtheorem{theorem}{Theorem}[section]

\newtheorem{lemma}[theorem]{Lemma}
\newtheorem{proposition}[theorem]{Proposition}

\newtheorem{remark}{Remark}
\allowdisplaybreaks \numberwithin{remark}{section}

\newcommand{\Real}{\mathbb{R}}

\newcommand{\IntK}{\mathring{K}}
\def\ot{\overline{\theta}}
\def\Kmz{\dot{K}}
\def\pK{\partial K}
\def\cLX{\mathcal{L}(X)}
\begin{document}
\title{A Generalized Krein-Rutman Theorem }

 \author{Lei Zhang\\
 {\small \it Department of Mathematics, University of Science and Technology of China,}\\
{\small \it Hefei, Anhui 230026, China}\\
{\small\it  Department of Mathematics and Statistics, Memorial University of Newfoundland,}\\
{\small \it  St. John's, NL AIC5S7, Canada} \\
{\small Email: \tt mathzhl@mail.ustc.edu.cn}
}

\date{}
\maketitle

\begin{abstract}
A generalized Krein-Rutman theorem for a strongly positive bounded linear operator whose spectral radius is larger than essential spectral radius is established: the spectral radius of the operator is an algebraically simple eigenvalue with strongly positive eigenvector and other eigenvalues are less than the spectral radius.
\end{abstract}

\section{Introduction}
\noindent 

Krein-Rutman theorem is a fundamental theorem in positive compact linear operator theory. It has been widely applied to Partial Differential Equations, Dynamical systems, Markov Process, Fixed Point Theory, and Functional Analysis. For instance, Krein-Rutman theorem is a basic tool to derive the existence of principal eigenvalue of a second order elliptic equation, which can be used to dynamic behaviors analysis of the corresponding system.  

In the pioneering works of Perron \cite{perron1907theorie} and Frobenius  \cite{frobenius1908ueber,frobenius1912matrizen}, it was proved that the spectral radius of a nonnegative square matrix is an eigenvalue with a nonnegative eigenvector. Krein and Rutman developed Perron and Frobenius's theory to a positive compact linear operator, which is the celebrated Krein-Rutman theorem. Nussbaum\cite{nussbaum1981eigenvectors} showed a more generalized Krein-Rutman theorem for a positive bounded linear operator whose spectral radius is larger than essential spectral radius.

Furthermore, Perron-Frobenius theorem also show that if the square matrix is nonnegative and irreducible, then the spectral radius is an algebraically simple eigenvalue with strongly positive eigenvector and other eigenvalues are less than the spectral radius. Krein-Rutman theorem also present a similar conclusion for a strongly positive compact linear operator(See, e.g., \cite[Theorem 1.2]{du2006order} or \cite[Theorem 19.3]{deimling1985nonlinear}). While for a strongly positive bounded linear operator whose spectral radius is larger than essential spectral radius, Nussbaum didn't show the similar conclusion in \cite{nussbaum1981eigenvectors} or other works. In this paper, we focus on proving this by using Krein-Rutman theorem proving thought ( \cite[Theorem 1.2]{du2006order}) and the observation about relationship of the essential spectral radius in different spaces.

To describe more exactly, now we recall some basic notations. Let $Y$  be a  Banach space. $P\subset Y$ is called a cone if $P$ is convex closed set, $\lambda P\subset P$ for $\lambda \geq 0$ and $P\cap (-P)=\emptyset$. Let $\mathring{P}$ be the interior of $P$, $\partial P=P\setminus \mathring{P}$ the boundary of $P$ and $ \dot{P}=P\setminus \{0\}$. Recall that $a \geq b$ if $a-b \in P$, $a>b$ if $a-b \in \dot{P}$, and $a\gg b$ if $a-b \in \mathring{P}$.  $P$ is called total if $Y=\overline{P-P}$. Obviously, if $\mathring{P}\neq \emptyset$, then $P$ is total.  Let $\mathcal{L}(Y)$ be the collection of all bounded linear operators from $Y$ to $Y$. We recall $T\in \mathcal{L}(Y)$ is positive operator if $T:P \rightarrow P$, and $T\in \mathcal{L}(Y)$ is strongly positive operator if $T:\dot{P} \rightarrow \mathring{P}$. Let $\sigma(T)$, $\sigma_e(T)$ denote the spectrum and essential spectrum of $T \in \cLX$, respectively, whose radius are denoted by $r(T)$ and $r_e(T)$. The following is the Krein-Rutman theorem(See, e.g., \cite[Theorem 1.2]{du2006order} or
\cite[Theorem 19.3]{deimling1985nonlinear}). 
\begin{theorem}[Krein-Rutman Theorem]\label{lem:KR}  
Let $X$ be a Banach space, a total cone $K \subset X$, and $L\in \cLX$ a compact positive operator with $r(L)>0$, then $r(L)$ is an eigenvalue with eigenvector $x \in \dot{K}$. If further, the interior $\IntK\neq \emptyset$ and $L$ is a strongly positive, then $x \in \mathring{K}$, $r(L)$ is an algebraically simple eigenvalue of $L$,  and $\vert\lambda \vert <r(L)$ for any other eigenvalue $\lambda$ of $L$.
\end{theorem}
The generalized Krein-Rutman theorem(weak version) is from \cite[Corollay 2.2]{nussbaum1981eigenvectors}.
\begin{theorem}[Generalized Krein-Rutman Theorem, a weak version]\label{thm:GKR:weak} Let $X$ be a Banach space, a total cone $K \subset X$ and $L\in\cLX$ a positive operator with $r(L)>r_e(L)$, then there exists a $x \in \dot{K}$ such that $Lx=r(L)x$.
\end{theorem}
 In this paper, we focus on proving the following generalized Krein–Rutman Theorem(strong version).
\begin{theorem}[Generalized Krein-Rutman Theorem, a strong version]\label{thm:GKR:strong}  Let $X$ be a Banach space, a cone $K \subset X$ with $\IntK\neq \emptyset$, and $L\in \cLX$ a strongly positive operator with $r(L)>r_e(L)$, then $r(L)$ is an algebraically simple eigenvalue of $L$ with an eigenvector $x \in \IntK$, and $\vert\lambda \vert <r(L)$ for any other eigenvalue $\lambda$ of $L$.
\end{theorem}

\section{Proof of a generalized Krein-Rutman theorem}


At first, we recall a formula to compute essential spectrum  (See e.g. \cite[Theorem 9.9]{deimling1985nonlinear}).
\begin{lemma}\label{lem:for:ess}
Let $X$ be a Banach space and $L\in \cLX$, then
$r_e(L)=\lim_{n\rightarrow +\infty}(\gamma(L^n))^\frac{1}{n}$, where 
\begin{equation}\label{equ:gamma:L}
\gamma(L):=\inf \{k>0: \alpha(L B) \leq k \alpha(B)\text{ for any bounded }B \subset X  \}.
\end{equation}
and $\alpha$ represent Kuratowski measure of non-compactness.
\end{lemma}

\begin{lemma}\label{lem:ess:L1<L}
Let $X$ be a Banach space, $X_1$ a subspace of $X$, $L\in \cLX$ and $LX_1\subset X_1$. Then $r_e(L_1)\leq r_e(L)$ where $L_1=L\vert_{X_1}$.
\end{lemma}
\begin{proof}
We set 
$$
A:=
\{k>0: \gamma(L B) \leq k \gamma(B)\text{ for any bounded }B \subset X  \},
$$
$$
A_1:=
\{k>0: \gamma(L_1 B_1) \leq k \gamma(B_1)\text{ for any bounded }B_1 \subset X_1  \}.
$$
Due to $L_1 B_1=L B_1$ for any bounded $B_1 \subset X_1 \subset X$, it implies that $k \in A_1$ for any $k \in A$. We deduce that $\gamma(L_1)\leq \gamma(L)$ from \eqref{equ:gamma:L}. By the same arguments, we have $\gamma(L_1^n)\leq \gamma(L^n)$ for any integer $n\geq 1$. It follows that $r_e(L_1) \leq r_e(L)$ from applying Lemma \ref{lem:for:ess}.
\end{proof}

\begin{lemma}\label{lem:pos:any_t}
Let $X$ be a Banach space, a cone $K \subset X$, and $x \in K$, if $x +t v \in K$ for any $t > 0$, then $v \in K$. 
\end{lemma}
\begin{proof}
By virtue of $t^{-1} x+v \in K$ for any $t> 0$, letting $t\rightarrow +\infty$, it follows that $v \in K$.
\end{proof}
\begin{lemma} \label{lem:str_pos:thre}
Let $X$ be a Banach space, a cone $K \subset X$ with $\IntK\neq \emptyset$. For $x \in \IntK$ and $v \notin K$, there is a finite number $t_0>0$ such that $x+t_0 v \in \pK$, $x+t v \in \IntK$ for $ 0\leq t < t_0$ and $x+t v \notin K$ for $  t> t_0$. 
\end{lemma}
\begin{proof}
Let $$t_0:=\sup \{t\geq 0: x+tv \in K\}.$$
It follows that that $t_0>0$ and $t_0<+\infty$ from $x \in \IntK$ and Lemma \ref{lem:pos:any_t}. By the definition of $t_0$, we deduce that $x+tv \notin K$ for any $t >t_0$.

And for any $t\in [0,t_0)$, we have $x+t v=\frac{t}{t_0}(x+t_0v) +\frac{t_0-t}{t_0}x \in \IntK$.
\end{proof}
\begin{lemma}\label{lem:eig:no_other>0}
Let $X$ be a Banach space, a cone $K \subset X$ with $\IntK\neq \emptyset$, $L\in \cLX$ a positive operator. If $r$ is an eigenvalue of $L$ corresponding an eigenvector $x \in \IntK$, then $L$ cannot have an eigenvalue $s>r$.
\end{lemma}
\begin{proof}
We prove this lemma by indirect argument. Suppose that there exists $s>r$ an eigenvalue of $L$ corresponding an eigenvector $v \notin K$ (otherwise change $v$ to $-v$). Due to $x \in \IntK$, we can find a $t_0>0$ such that $x+t_0 v \in \pK$ and $x+t v \notin K$ for $t > t_0$  from applying Lemma \ref{lem:str_pos:thre}. Then $$L(x+t_0 v)=r x +t_0 s v=r(x+t_0 \frac{s}{r}v) \in K.$$
Therefore $s \leq r$, a contradiction with $s > r$.
\end{proof}
\begin{remark}
If $x\in K$ but $x \notin \IntK$, the conclusion of Lemma \ref{lem:eig:no_other>0} maybe not right. For example, $1$ is an eigenvalue of matrix
$\left(\begin{array}{cc}
2 & 0 \\ 
0 & 1
\end{array}\right)$ corresponding a positive eigenvector 
$\left(\begin{array}{c}
0  \\ 
1 
\end{array}\right)$, while $2>1$ is also an eigenvalue of this matrix.
\end{remark}
\begin{remark}
There exists a positive operator rather than strongly positive operator satisfying the condition of Lemma \ref{lem:eig:no_other>0}, for instance,
$\left(\begin{array}{cc}
2 & 0 \\ 
1 & 1
\end{array}\right) $.
\end{remark}
\begin{lemma} \label{lem:simple}
Let $X$ be a Banach space, a cone $K \subset X$ with $\IntK\neq \emptyset$, and $L\in \cLX$ a strongly positive operator.
If $r>0$ is an eigenvalue of $L$ corresponding an eigenvector $x \in \Kmz$, then $x \in \IntK$ and $r$ is an algebraically simple eigenvalue.
\end{lemma}
\begin{proof}
It follows that $x=r^{-1}Lx \in \IntK$ from $L$ is strongly positive. 

We first show that $r$ is geometrically simple. Assume there is $v\neq 0$ and $v\notin K$ (otherwise change $v$ to $-v$) such that $L v = r v$. 

We claim that $x + t v \notin \pK$ unless $x+ tv =0$. 
If the claim is not right, then there is $t_0\neq 0$ such that $x + t_0 v \in \pK$ and $x + t_0 v \neq 0$. Since $L$ is strongly positive and $L(x+ t v)=r(x+ t v) \text{ for all  } t \in \Real$, it implies that $x+ t_0 v=r^{-1}L(x+ t_0 v) \in \IntK$. This contradiction proves the claim.

 By Lemma \ref{lem:str_pos:thre}, there exists $t_1> 0$ such that $x+t_1 v \in \pK$ as $v \notin K$. We have $x+t_1 v=0$ from applying the above claim. This proves $r$ is a geometrically simple eigenvalue of $L$.	

Next, we want to show that $r$ is algebraically simple.
Let $(rI -L)^2 v=0$, it follows that $rv-Lv=t_0 x$ for some $t_0$ as $r$ is a geometrically simple eigenvalue. It is necessary to show that $t_0 =0$. Assume $t_0> 0$(otherwise change $v$ to $-v$). We are going to show that $v \in K$. Otherwise, there is a $s_0>0$ such that $x+s_0v\in \pK$ by Lemma \ref{lem:str_pos:thre}. But from  $L(x+s_0v)=r(x+s_0v)-s_0t_0x \ll r(x+s_0v)$, we conclude $x+s_0v \in \IntK$. This contradiction shows that $v \in K$. Therefore $v=r^{-1}(Lv+t_0 x) \in \IntK$.
Lemma \ref{lem:str_pos:thre} implies that there exists $t_1>0$ such that $v-t_1x \in \pK$ as $-x \notin K$. But from 
$rv-t_0x-t_1rx=L(v-t_1x)\geq 0$, we deduce $r(v-t_1x)\geq t_0x\gg0$, contradicting $v-t_1x \in \pK$. Hence, we must have $t_0=0$ and $rv-Lv=0$. This proves that $r$ is an algebraically simple eigenvalue of $L$.
\end{proof}
\begin{proof}[Proof of the Theorem \ref{thm:GKR:strong}]
By Theorem \ref{thm:GKR:weak} and Lemma \ref{lem:simple}, it implies that $r(L)$ is an algebraically simple eigenvalue of $L$ with an eigenvector $ x \in \IntK$. Let $\lambda \neq r(L)$ be an eigenvalue of $L$ with eigenvector $w$, we want to prove $\vert \lambda \vert <r(L)$.

If $\lambda>0$, it is a straightforward result of Lemma \ref{lem:eig:no_other>0}.

If $\lambda<0$, then from $L^2 x=r(L)^2x$, $L^2 w= \lambda^2 w$ and the above argument(applied to $L^2$), we deduce $\vert \lambda\vert ^2 <r(L)^2$  and hence $\vert\lambda \vert < r(L) $.

Now we consider $\lambda= \sigma +i \tau$ with $\tau \neq 0$ and suppose $\vert \lambda \vert =r(L)$. Then necessarily  $w=u +i v$ and 
\begin{equation}\label{equ:uv}
Lu = \sigma u - \tau v,\qquad 
Lv=\tau u+\sigma v
\end{equation}
We observe that $u$, $v$ are linearly independent for otherwise we have $\tau=0$. Let $X_1= span\{u,v\}$. Then \eqref{equ:uv} implies that $X_1$ is an invariant subspace of $L$. 
\\
Claim 1: $K_1:=K \cap X_1 =\{0\}.$

If the claim is not right, then $K_1$ is a positive cone in $X_1$ with nonempty interior, as for any $w \in \dot{K_1} $, $Lw \in X_1 \cap \IntK=\mathring{K_1}$. Now we denote $L_1:=L\vert_{X_1}$. It follows that $r(L_1)\geq r(L)$ from $\vert\lambda\vert =r(L)$ and $\lambda$ is an eigenvalue of $L_1$. In view of the  Lemma \ref{lem:ess:L1<L}, we conclude $r_e(L_1)\leq r_e(L)$. Hence $r(L_1)\geq r(L)>r_e(L)\geq r_e(L_1)$. By applying Theorem \ref{thm:GKR:weak} again to $L_1$ on $X_1$, there exists $w_0 \in \dot{K_1}$ such that $L_1w_0=r(L_1)w_0$, which implies that $r(L_1)=r(L)$ by Lemma \ref{lem:eig:no_other>0}. Since $r(L)$ is an algebraically simple eigenvalue of $L$, then $w_0 \in span \{x\}$. In other words, $x=\alpha u+ \beta v$ for some real number $\alpha$ and $\beta$. But one can use \eqref{equ:uv} and $Lx=r(L)x$ to derive $\alpha=\beta=0$, a contradiction. Therefore $K_1=\{0\}$.

Claim 2: the set
$\Sigma:= \{(\xi,\eta)\in \Real^2:x+\xi u +\eta v \in K\} $
is bounded and closed.

For any $(\xi,\eta)\in \Real^2$ with $\xi^2+\eta^2\neq 0$, there is a unique $R>0$ and $\theta \in [0,2\pi)$ such that $(\xi,\eta)=R(\cos\theta,\sin\theta)$. We denote $w(\theta)=u\cos\theta +v\sin \theta$ for any $\theta\in \Real$. $x \in \IntK$ derives that $\Sigma= \{R(\cos\theta,\sin \theta)\in \Real^2:x+ R w(\theta) \in K,R\geq 0, \theta\in\Real\}$. We conclude that $w(\theta)\notin K$ for any $\theta\in\Real$ by Claim 1.
By Lemma \ref{lem:str_pos:thre}, it follows for any $\theta \in \Real$, there is $R_0(\theta)>0$ such that $x+R_0(\theta)w(\theta)\in \pK$, $x+Rw(\theta) \in \IntK$ for any $R\in [0,R_0(\theta))$ and $x+Rw(\theta) \notin K$ for any $R>R_0(\theta)$. 
Hence, to prove the $\Sigma$ is bounded,
it is enough to prove that $R_0(\theta)$ is bounded on $\Real$.
 
 Now, we are going to show that $R_0(\theta)$ is continuous on $\Real$. Suppose $R_0(\theta)$ is not upper semi-continuous at some point $\ot \in \Real$, that is, there is a sequence $\theta_n \rightarrow \overline{\theta}$ as $n \rightarrow +\infty$ and $\epsilon_0>0$, such that  $R_0(\theta_n)\geq R_0(\overline{\theta})+\epsilon_0$ for $n$ large enough. Since $w(\theta)$ is continuous on $\Real$ and $ x+R w(\theta_n) \in K$ for any $R\in [0,R_0(\theta_n)]$ and  $n\geq 1$, it implies that $x+(R_0(\ot)+\epsilon_0) w(\ot) \in K$, which is a contradiction with $x+Rw(\ot) \notin K$ for any $R>R_0(\ot)$. Therefore $R_0(\theta)$ is upper semi-continuous on $\Real$. By the similar arguments, $R_0(\theta)$ is lower semi-continuous on $\Real$. Hence $R_0(\theta)$ is continuous.
Since $w(\theta)=w(\theta+2\pi)$ for any $\theta \in \Real$, we derive  $R_0(\theta)=R_0(\theta+2\pi)$ for any $\theta \in \Real$. Therefore $R_0(\theta)$ is bounded on $\Real$ and  $\Sigma$ is bounded.
It follows that $\Sigma$ is closed from $K$ is closed. The claim is proved.

Claim 2 implies that $M:=\sup \{\xi^2+\eta^2:(\xi,\eta)\in \Sigma\}>0$ and $M$ is achieved at some point $(\xi_0,\eta_0)\in \Sigma$. Let $z_0=x+\xi_0 u+ \eta_0 v$, then $z_0 \in \Kmz$ as $K_1=\{0\}$. Hence, $Lz_0 \in \IntK$ and there is $\alpha >0$ small enough such that $L z_0 \geq \alpha x$, that is,
\begin{equation}\label{equ:xi1&eta1}
(r(L)-\alpha)x + \xi_1 u+ \eta_1 v \geq 0,
\end{equation}
where $\xi_1=\xi_0\sigma+ \eta_0 \tau, \eta_1=\eta_0 \sigma-\xi_0 \tau$. Clearly $\xi_1^2 +\eta_1^2=(\sigma^2+\tau^2)(\xi_0^2+\eta_0^2)=M\vert\lambda 
\vert^2$.
By \eqref{equ:xi1&eta1}, we find that $(r(L)-\alpha)^{-1}(\xi_1,\eta_1)\in \Sigma$. Hence $\xi_1^2+\eta_1^2\leq M(r(L)-\alpha)^2$. We deduce that $\vert \lambda \vert ^2 \leq (r(L)-\alpha)^2$, which is a contradiction with $\vert \lambda \vert =r(L)$. Hence, we proved $\vert \lambda \vert <r(L)$.
\end{proof}
\begin{proposition}
Let $r(L)$ be an eigenvalue of $L$ with eigenvector $x \in \IntK$. For $y \in \Kmz$, we consider the following equation: 
\begin{equation}\label{equ:thre}
(\lambda I-L)z=y,
\end{equation}
(i)   ~~If $\lambda>r(L)$, the equation \eqref{equ:thre} has a unique solution $z\in \IntK$.  \\
(ii)  ~If $\lambda \leq r(L)$, the equation \eqref{equ:thre}  has no solution in $K$ .\\
(iii) If $\lambda = r(L)$, the equation \eqref{equ:thre}  has no solution.
\end{proposition}
\begin{proof}
It follows  that $z\neq 0$ from $ (\lambda I- L)z=y$.

$(i)$
If $\lambda>r(L)$, then $z=(\lambda I- L)^{-1}y\in \IntK$.

$(ii)$ Suppose $z \in K$. If $\lambda\leq 0$, then $y=(\lambda I-L)z\leq -Lz$, a contradiction.
Now we assume $0 <\lambda \leq r(L)$. It follows that $z=\lambda^{-1}(y +Lz) \in \IntK$. Lemma \ref{lem:str_pos:thre} implies that there exists $t_0$ such that $z-t_0 x \in \pK$ as $-x\notin K$. By $L(z-t_0x)=\lambda z -y -t_0 r(L)x$, we deduce $z-t_0x=r(L)^{-1}(L(z-t_0 x)+y+(r(L)-\lambda) z) \in \IntK$, which  contradicts with $z-t_0 x \in \pK$.

$(iii)$
If $ (r(L)I- L)z=y$, there exists $t_1>0$ small enough such that $x +t_1 z\in \IntK$ as $x\in \IntK$. Due to $(r(L)I- L)x=0$, it implies that $ (r(L)I- L)(x+t_1 z)=t_1 y \in K$, which contradicts with $(ii)$.
\end{proof}
\section*{Acknowledgements}We would like to thank Prof. Xiao-Qiang Zhao and Prof. Xing Liang for their helpful discussions and advice during the preparation of this work. We would also like thank Prof. Xue-Feng Wang for his short courses about  Krein-Rutman theorem.
\bibliographystyle{siamplain}

\end{document}